\documentclass[11pt]{article}
\usepackage[utf8]{inputenc}
\usepackage{tikz}
\usepackage{fullpage}
\usepackage{enumitem}
\usepackage{hyperref}
\usepackage{xcolor}
\usepackage{amssymb,amsthm,amsmath}
\usepackage{color}
\usepackage{lineno}
\usepackage[nameinlink]{cleveref}

\newtheorem{theorem}{Theorem}
\newtheorem{lemma}[theorem]{Lemma}
\newtheorem{claim}[theorem]{Claim}
\newtheorem{corollary}[theorem]{Corollary}
\newtheorem{conjecture}{Conjecture}
\newtheorem{proposition}[theorem]{Proposition}
\theoremstyle{definition}
\newtheorem{ques}{Question}
\newtheorem{const}{Construction}

\newtheoremstyle{case}{}{}{\normalfont}{}{\itshape}{\normalfont:}{ }{}

\theoremstyle{case}

\DeclareMathOperator{\wsat}{wsat}
\DeclareMathOperator{\vspan}{span}
\DeclareMathOperator{\vdim}{dim}
\newcommand{\vc}{U}
\newcommand{\vb}{\textbf{u}}

\title{Weak saturation numbers of complete \\bipartite graphs in the clique}
\author{Gal Kronenberg \thanks{Mathematical Institute, University of Oxford, Oxford, UK. E-mail: \texttt{kronenberg@maths.ox.ac.uk}.} \and Ta\'isa Martins \thanks{Instituto de Matem\'atica, Universidade Federal Fluminense, Niter\'oi, Brazil. \newline E-mail: \texttt{tlmartins@id.uff.br}.} \and Natasha Morrison\thanks{Department of Mathematics and Statistics, University of Victoria, David Turpin Building,
3800 Finnerty Road, Victoria, B.C., Canada V8P 5C2. \newline E-mail: \texttt{nmorrison@uvic.ca}.}}
\date{}

\begin{document}
\maketitle

\begin{abstract}
The notion of weak saturation was introduced by Bollob\'as in 1968.
    Let $F$ and $H$ be graphs. A spanning subgraph $G \subseteq F$ is \emph{weakly $(F,H)$-saturated} if it contains no copy of $H$ but there exists an ordering $e_1,\ldots,e_t$ of $E(F)\setminus E(G)$ such that for each $i \in [t]$, the graph $G \cup \{e_1,\ldots,e_i\}$ contains a copy $H'$ of $H$ such that $e_i \in H'$. Define $\wsat(F,H)$ to be the minimum number of edges in a weakly $(F,H)$-saturated graph. In this paper, we prove for all $t \ge 2$ and $n \ge 3t-3$, that $\wsat(K_n,K_{t,t}) = (t-1)(n + 1 - t/2)$, and we determine the value of $\wsat(K_n,K_{t-1,t})$ as well. For fixed $2 \le s < t$, we also obtain bounds on $\wsat(K_n,K_{s,t})$ that are asymptotically tight.
    \end{abstract}

\section{Introduction}

 Let $F$ and $H$ be graphs. A spanning subgraph $G$ of $F$ is said to be \textit{weakly $(F,H)$-saturated}, if
$G$ contains no copies of $H$, but there exists an
ordering $e_1,\ldots, e_t$ of $E(F) \setminus E(G)$ such that the addition of $e_i$ to $G \cup \{e_1,\dots , e_{i-1}\}$
creates a new copy $H'$ of $H$ where $e_i \in H'$, for every $i \in [t]$.
The \textit{weak saturation number} of $H$ in $F$, is defined to be
$$\wsat(F,H):= \min \{|E(G)|: G \text{ is weakly }(F,H)\text{-}\text{saturated}\}.$$
That is, $\wsat(F,H)$ is the minimum number of edges of a weakly $(F, H)$-saturated graph. In the most natural case where $F$ is the complete graph on $n$ vertices, denoted $K_n$, we write $\wsat(n,H):=\wsat(K_n,H)$.

Weak saturation of graphs was initially introduced by Bollob\'as~\cite{BollobasWSat} in 1968 and has grown to a substantial area of research. Originally motivated by the problem of determining the saturation\footnote{A spanning subgraph $G \subseteq F$ is $(F,H)$-saturated, if it contains no copy of $H$ but the addition of any edge of $E(F)\setminus E(G)$ creates a copy of $H$.} number of $k$-uniform hypergraphs, Bollob\'{a}s determined $\wsat(n,K_m)$ for $3 \le m \le 7$ and conjectured that the graph obtained by removing a copy of $K_{n-r+2}$ from $K_n$ is best possible for $\wsat(n,K_r)$. Using a very elegant generalisation of the Bollob\'{a}s two families theorem~\cite{bol2f}, Lov\'asz~\cite{LovaszMultiLinear} was the first to confirm this conjecture.
\begin{theorem}[Lov\'asz~\cite{LovaszMultiLinear}]\label{thm:clique}
Let $n \ge r \ge 2$. Then
$$\wsat(n,K_r) = \binom{n}{2} - \binom{n-r+2}{2}.$$
\end{theorem}

Independent proofs were later found by Alon \cite{AlonWeak1985}, Frankl~\cite{FranklWeakSat1982}, and Kalai~\cite{ KalaiWeakSat1985,KalaiWeakSat1984}. Interestingly, all these proofs utilise algebraic techniques and no combinatorial proof of Theorem~\ref{thm:clique} is known. 

Let $K_{s,t}$ denote the complete bipartite graph with vertex classes size $s$ and $t$. In this article, we study the next most natural question to consider regarding weak saturation: What is $\wsat(n,K_{s,t})$? In the case of the balanced complete bipartite graph we determine this number exactly. Our first theorem is the following.
\begin{theorem}\label{thm:KttWsat}
Let $t \ge 2$ and $n \ge 3t-3$. Then $$\wsat(K_n,K_{t,t}) = (t-1)(n+1 - t/2).$$
\end{theorem}
Given this theorem, a short argument yields an exact result for $K_{t,t+1}$.
\begin{corollary}\label{corr:Ktt+1}
	Let $t \ge 2$ and $n \ge 3t-3$. Then $$\wsat(K_n, K_{t,t+1}) = (t-1)(n + 1 - t/2) +1.$$
\end{corollary} 

For when $s<t$, we also obtain a general bound for $K_{s,t}$, which is tight asymptotically for fixed $s,t$ and $n$ large. 

\begin{theorem}\label{thm:genst}
Let $2 \le s < t$ and $n \ge 4t$. Then
$$\wsat(n,K_{s,t}) = n(s-1) + c(s,t),$$
where $c(s,t)$ is an integer depending only on $s$ and $t$.
\end{theorem}
Surprisingly, despite the large body of work and number of alternate proofs to determine $\wsat(n,K_m)$, prior to this work very little was known about $\wsat(n,K_{s,t})$. The case when $s=2$ and $t=3$,  was shown to be $n+1$, by Faudree, Gould and Jacobson~\cite{FaudreeWsatSparseGraphs}, who also determined the weak saturation number for various families of sparse graphs. A trivial lower bound of $n \cdot (s-1)/2$ can be obtained by observing that every vertex in a weakly $(K_n,K_{s,t})$-saturated graph must have degree at least $(s-1)$, when $s\leq t$. But other than this, no general lower bound was previously known. 

A more well studied setting is where the weak saturation process takes place inside a bipartite ambient graph (i.e.~$H$ is bipartite). In~\cite{AlonWeak1985}, Alon studied a labelled version of this problem called \emph{bisaturation}.
For a bipartite graph $H=(V_1\cup V_2,E)$, we say that a spanning subgraph $G\subseteq K_{\ell,m}$ is weakly $(K_{\ell,m},H)$-bisaturated, if there exists an
ordering $e_1,\ldots,e_t$ of $E(K_{\ell,m}) \setminus E(G)$ such that the addition of $e_i$ to $G \cup \{e_1,\dots , e_{i-1}\}$ will
create a new copy $H'$ of $H$ where $e_i \in H'$, with $V_1$ in the first class and $V_2$ in the second class, for every $i \in [t]$. The \emph{bisaturation number}, denoted $w(\ell,m,H)$, is the minimal possible number of edges in a weakly $H$-bisaturated graph. 

The bisaturation number is closely related to the weak saturation number inside a bipartite graph. Indeed, as every weakly $(K_{\ell,m},H)$-bisaturated graph $G$ is also weakly $(K_{\ell,m},H)$-saturated, we have
 $w(\ell,m,H)\geq \text{wsat}(K_{\ell,m},H)$.  In addition, when $s<t$ we also have $w(s,n-s,K_{s,t})= \text{wsat}(K_{s,n-s},K_{s,t})$, and when $\ell,m\geq t$, we have $w(\ell,m,K_{t,t})= \text{wsat}(K_{\ell,m},K_{t,t})$. The value of $w(\ell,m,K_{r,t})$ was determined precisely by Alon,  who also proved a generalization for hypergraphs.
 
 \begin{theorem}[Alon~\cite{AlonWeak1985}] 
 For $2\leq s\leq t$  and $2\leq \ell \leq m$, we have $$w(\ell,m,K_{s,t})=\ell\cdot m-(\ell-s+1)(m-t+1).$$
 \end{theorem}
 
Moshkovitz and Shapira~\cite{MoshShapira} showed how to deduce $\wsat(K_{n,n},K_{s,t})$ from this result, giving the following theorem.
\begin{theorem}[Moshkovitz and Shapira~\cite{MoshShapira}]\label{thm:MS}
Let $2 \le s \le t \le n$. Then
$$\wsat(K_{n,n},K_{s,t})=n^2-(n-s+1)^2+(t-s)^2.$$
\end{theorem}
We would like to note that the main contribution in \cite{MoshShapira} is studying an analogous process in multipartite hypergraphs. In doing so, they also prove a very beautiful two-families type theorem.

The proof of Theorem~\ref{thm:MS} can be easily generalised to determine $\wsat(K_{\ell,m},K_{s,t})$.
\begin{theorem}\label{thm:wsat-imp}
Let $2 \le s \le \ell, t \le m$. Then
$$\wsat(K_{\ell,m},K_{s,t})=(m+\ell-s+1)(s-1)+(t-s)^2.$$
\end{theorem}

For completeness, we include the full argument in \Cref{appen:BipInBip}.  
Observe that when $\ell + m = n$, Theorem~\ref{thm:wsat-imp} implies that $\wsat(K_{\ell,m},K_{t,t})=(t-1)(n+1-t)$. We note that this, along with Theorem~\ref{thm:KttWsat} gives the following relationship between weak saturation numbers of complete balanced bipartite graphs in the clique and in complete bipartite graphs.
\begin{corollary}\label{cor:rel}
For $t \ge 2$, $n \ge 3t-3$ and $\ell,m \ge 2$ such that $\ell + m = n$, we have
$$\wsat(n,K_{t,t}) = \wsat(K_{\ell,m},K_{t,t}) + \binom{t}{2}.$$
\end{corollary}

In particular, taking a construction of minimum size for  $\wsat(K_{t,n-t},K_{t,t})$ (see \Cref{appen:BipInBip}) and adding a copy of $K_t$ in the larger part of the graph, will give a minimum example for $\wsat(n,K_{t,t})$.
This argument will give an upper bound also for the unbalanced case, that is, the upper bound $\wsat(n,K_{s,t})\leq \wsat(K_{\ell,m},K_{s,t})+\binom t2$ is always true for $m+\ell=n$, however it is tight only for the balanced case. Indeed, for $s<t$, the proof of \Cref{thm:genst} (more specifically, \Cref{prop:KstWsatUpper}) will give a better upper bound.
%

Weak saturation numbers have been also studied in a large number of other settings. Amongst others are the weak saturation numbers for complete multipartite graphs and hypergraphs \cite{AlonWeak1985}, asymptotics of the weak saturation number of hypergraphs \cite{Tuza92}, the weak saturation number for families
of hypergraphs with a fixed number of edges \cite{EFT91,PikhurkoWeakSatHypergraphs1,Tuza88}, complete bipartite hypergraphs (in complete bipartite hypergraphs) \cite{BBMR,MoshShapira}, pyramids in hypergraphs \cite{PikhurkoWeakSatHypergraphs}, families of graphs in the complete graph \cite{Sem97,BS02,Sid07}, families of disjoint copies of graphs \cite{FaudreeWsatMultipleCopies}, and the case that $H$ is the hypercube or the grid \cite{BBMR,BaloghPete1998,MorrisonNoelHypercube}. For a short survey see \cite[Section~10]{FaudreeSurvey}.\\

The paper is organized as follows. In Section~\ref{sec:pre} we introduce an important tool that will be used to prove the lower bound in Theorem~\ref{thm:KttWsat}, which is then proved in Section~\ref{sec:Ktt}. Theorem~\ref{thm:genst} is proved in Section~\ref{sec:Kst}. We conclude in Section~\ref{sec:con} by discussing some generalisations and interesting open problems.\\

\textbf{Note added after publication:} It was recently brought to our attention that  \Cref{thm:KttWsat} was independently proved in 1985 by Kalai~\cite[Theorem 9.1]{KalaiWeakSat1985} using matroids.

\section{Preliminaries}\label{sec:pre}
In order to exactly determine a particular weak saturation number, a common strategy is to prove matching upper and lower bounds (often in very different ways). To prove an upper bound of $M$, it suffices to find a construction of a graph with $M$ edges that is weakly $(F,H)$-saturated. Indeed, this is precisely what we do in the proof of Theorem~\ref{thm:KttWsat} (see Lemma~\ref{lem:upp}). However, in order to prove a lower bound of $M$, we must show that \emph{no} graph on $M-1$ edges can be weakly $(F,H)$-saturated. 

In order to do this we will utilise the following lemma.

\begin{lemma}[Balogh, Bollob\'as, Morris and Riordan~\cite{BBMR}]\label{LinearAlgebraBootstrapPer}
Let $F$ and $H$ be graphs and let $W$ be a vector space. Suppose that there exists a set
$\{f_e : e \in E(F)\} \subseteq W$ such that for every copy $H'$ of $H$ in $F$ there are non-zero scalars $\{c_{e,H'} : e \in E(H')\}$
such that $\sum_{e \in E(H')} c_{e,H'}f_e = 0$. Then
$$\wsat(F, H) \ge \vdim (\vspan \{f_e : e \in E(F)\}).$$
\end{lemma}

The proof of Lemma~\ref{LinearAlgebraBootstrapPer} is short and beautiful, and so we include it here for the reader's enjoyment.

\begin{proof}
Let $F_0$ be any weakly $(F,H)$-saturated graph. Let $e_1,\ldots,e_m$ be an ordering of $E(F)\setminus E(F_0)$ such that for each $i \in [m]$, the graph $F_i:= F_0 \cup \{e_1,\ldots,e_i\}$ contains a copy $H_i$ of $H$, such that $e_i \in H_i$. Now, by hypothesis, for each $i \in [m]$, we have
$$f_{e_i} \in \vspan \{f_e: e \in E(H_i)\setminus \{e_i\}\}.$$ 

As $H_i \subseteq  F_0 \cup \{e_1,\ldots,e_i\}$, this implies that, for each $i$,
$$f_{e_i} \in \vspan\{f_e: e \in E(F_0) \cup \{e_1,\ldots,e_{i-1}\}\}.$$

So, for all $i \in [m]$ we have
$$\vspan\{f_{e} : e \in E(F_{i-1})\} = \vspan\{f_{e} : e  \in E(F_{i})\}.$$
And so
\begin{align*}
    |E(F_0)| \ge \dim (\vspan \{f_e: e \in E(F_0)\}) = \dim(\vspan\{f_e: e \in E(F_m)\}) = \dim(\vspan\{f_e: e \in F\}),
\end{align*}
as required.
\end{proof}
We would like to briefly remark that the condition in the lemma of assigning vectors to edges in such a way that those on copies of $H$ satisfy a particular dependence is not in itself difficult to satisfy (for example, just put the same vector on every edge). However, doing this would result in a terrible lower bound on $\wsat(F,H)$ (as $\dim(\vspan\{f_e:e \in E(F)\} = 1)$). So the difficulty in applying this lemma lies in finding vectors that both satisfy the dependence condition \emph{and} have a large span. 

This lemma is very powerful, as it turns the problem of finding a lower bound for a weak saturation number into a \emph{constructive} problem. Indeed, one need only find a suitable vector space and assign certain vectors to the edges of $F$ to obtain the bound. In Lemma~\ref{lem:low}, we present a collection of vectors such that the vectors assigned to each copy of $K_{t,t}$ satisfy the required dependence property. We then show that the vectors assigned to a copy of our upper bound construction are linearly independent (and hence the lower bound matches the upper).

A more general version of Lemma~\ref{LinearAlgebraBootstrapPer} was originally used by Balogh, Bollob\'{a}s, Morris and Riordan~\cite{BBMR} in the study of a bootstrap process on hypergraphs. It has since been used to determine exact weak saturation numbers (see for example~\cite{MorrisonNoelHypergraph,MorrisonNoelHypercube}). We remark that in all these cases, and in the case of Theorem~\ref{thm:KttWsat} here, no purely combinatorial proof is known that gives any of the lower bounds that are proved via Lemma~\ref{LinearAlgebraBootstrapPer}. This is a common phenomenon in this area, where the literature is full of proofs via algebraic techniques for which no combinatorial proof is known. It would be very interesting to see a purely combinatorial proof of any of these weak saturation results.

\section{Proof of Theorem~\ref{thm:KttWsat}}\label{sec:Ktt}
We will prove Theorem~\ref{thm:KttWsat} in two steps. First in Lemma~\ref{lem:upp} we will prove the upper bound by exhibiting a weakly $(K_n,K_{t,t})$-saturated graph with $(t-1)(n+1 - t/2)$ edges. Then in Lemma~\ref{lem:low} we prove a matching lower bound by applying Lemma~\ref{LinearAlgebraBootstrapPer}. 

Let $G$ be a graph.  For disjoint vertex sets $X, Y \subseteq V(G)$, we denote by $G[X,Y]$ the bipartite graph induced by the edges of $G$ with one endpoint in $X$ and the other in $Y$.
\begin{const}\label{const:ktt}
Let $X$, $Y$ and $Z$ be disjoint subsets of $[n]$ of cardinality $t$, $t-1$ and $n-2t+1$, respectively. Define $G_n$ to be the graph on vertex set $X \cup Y \cup Z$ and $uv$ is an edge of $G_n$ if and only if either $u \in X \cup Z$ and $v \in Y$, or $u,v \in X$. That is, $X$ is a clique on $t$ vertices, $Y$ and $Z$ are both independent sets, and $G_n[X\cup Z,Y]$ is a complete bipartite graph.
\end{const}
See Figure~\ref{fig:Gn} for an illustration. 
Observe that $G_n$ has $(t-1)(n + 1 - t/2)$ edges. To prove the upper bound we will show that $G_n$ is weakly $(K_n,K_{t,t})$-saturated.
\begin{figure}[htbp]
\centering
\includegraphics[width=0.4\textwidth]{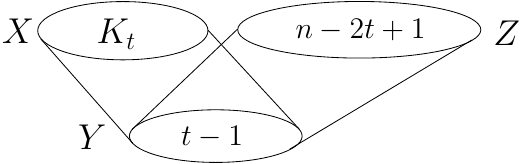}
\caption{The graph $G_n$. The sets $Y$ and $Z$ are independent sets of cardinality $t-1$ and $n-2t+1$, respectively. All edges between $X \cup Z$ and $Y$ are present. $X$ is a clique on $t$ vertices.}
\label{fig:Gn}
\end{figure}

\begin{lemma}\label{lem:upp} Let $t \ge 2$ and $n \ge 3t-3$.  Then
$\wsat(K_n, K_{t,t}) \le (t-1)(n + 1 - t/2).$
\end{lemma}
\begin{proof}
We will show that $G_n$ is weakly $(K_n,K_{t,t})$-saturated by adding the edges of $E(K_n) \setminus E(G_n)$ in such a way that the addition of each edge $e$ creates a copy of $K_{t,t}$ containing $e$.

Let $X,Y,Z$ be as defined above. We are first able to add all the edges between $X$ and $Z$ (in any order). Indeed, consider $e = xz$ for some $x \in X$ and $z \in Z$. Observe that $G_n[Y \cup \{x\}, X \setminus\{x\}\cup \{z\}]$ is a copy of $K_{t,t}\setminus \{e\}$ and hence $e$ can be added. 

We next show that we can add all edges within $Z$. Let $e = z_1z_2$, where $z_1, z_2 \in Z$ and let $X'$ be a subset of $X$ of size $t-1$. Observe that $G_n[Y \cup \{z_1\}, X' \cup \{z_2\}]$ is a copy of $K_{t,t}\setminus \{e\}$. Hence $e$ can be added. 

It remains to show that any edge $e = y_1y_2$, where $y_1,y_2 \in Y$ can be added. As all edges outside of $Y$ have been added and $n - t +1 \geq 2t -2$, any subset $S$ of $X \cup Z$ of cardinality $2t-2$ along with $y_1$ and $y_2$ contains a copy of $K_{t,t}\setminus \{e\}$. Hence every edge in $Y$ can be added. This completes the proof.
\end{proof}

Proving the lower bound is much more involved. Our strategy is to apply  Lemma~\ref{LinearAlgebraBootstrapPer}. We will construct a family of vectors $\{f_e: e \in E(K_n)\}$ such that: (1) for any copy $H$ of $K_{t,t}$ in $K_n$, the vectors $\{f_e: e \in E(H)\}$ have a non-trivial dependence; (2) the subset of vectors $\{f_e:e \in E(G_n)\}$ are linearly independent.

\begin{lemma}\label{lem:low}
Let $t \ge 2$ and $n \ge 3t-3$.  Then
$\wsat(K_n, K_{t,t}) \ge (t-1)(n + 1 - t/2).$
\end{lemma}

\begin{proof}
 Let $W$ be a vector space that is the direct sum of $n$ copies of $\mathbb{R}^{t-1}$, one for each vertex of $K_n$. That is,
$$W := \bigoplus_{v \in K_n}\mathbb{R}^{t-1}.$$
For each edge $e \in K_n$ we will associate a vector $f_e \in W$, in such a way that the hypotheses of Lemma~\ref{LinearAlgebraBootstrapPer} are satisfied.

For $v \in K_n$ and $w \in W$, let $\pi_v: W \rightarrow \mathbb{R}^{t-1}$, denote the projection of $w$ to the copy of $\mathbb{R}^{t-1}$ corresponding to $v$.
Let $U = \{\vb_v: v \in K_n\} \subseteq \mathbb{R}^{t-1}$ be a family of $n$ vectors in general position. That is, any $t-1$ vectors of $U$ are linearly independent (and hence any $t$ vectors have a unique dependence, up to scaling by a constant factor). Existence of such a family can be seen by picking $n$ random vectors from $\mathbb{R}^{t-1}$ and observing that, with high probability, any subset of size $t-1$ is independent.

To an edge $e = xy \in E(K_n)$ we will associate the vector $f_e \in \mathbb{R}^{n(t-1)}$, defined such that $\pi_x(f_e) = \vb_y$, $\pi_y(f_e) = \vb_x$ and $\pi_v(f_e) = 0^{t-1}$ (where we use this notation to represent the $(t-1)$-dimensional all zero vector), for all $v \in V(K_n) \setminus \{x,y\}$. Let $E \subseteq E(K_n)$ and $\{c_e : e \in E\}$ be a set of non-zero scalars. 
Observe that  
\begin{equation}\label{obs:coo}\sum_{e \in E} c_ef_e = 0 \hspace{0.2cm} \text{  if and only if  } \hspace{0.2cm} \pi_v\left( \sum_{e \in E} c_{e} f_e \right) =  \sum_{e \in E} c_{e}\pi_v(f_{e}) = 0, \text{ for every } v \in K_n. \end{equation}

Let us now affirm that the family $\{f_e: e \in E(K_n)\}$ satisfies the hypotheses of Lemma~\ref{LinearAlgebraBootstrapPer}.

\begin{claim}\label{cl:dep}
For every copy $H$ of $K_{t,t}$ in $K_n$, there exist non-zero scalars $\{c_{e,H} : e \in E(H)\}$ such that $\sum_{e \in E(H)} c_{e,H}f_e = 0$.
\end{claim}
\begin{proof}
For simplicity of notation, we will write $c_e$ for $c_{e,H}$. Assume without loss of generality that $V(H) = (\{v_1, \ldots, v_t\},\{ w_1, \ldots, w_t\})$. As $U$ is a family of vectors in general position, there exist non-zero scalars $\alpha_1 \ldots \alpha_t$ such that $\sum_{i= 1}^t \alpha_i \vb_{v_i} = 0$. Similarly, let  $\beta_1 \ldots \beta_t$ be non-zero scalars such that $\sum_{i= 1}^t \beta_i \vb_{w_i} = 0$. For $e = v_iw_j \in H$, define $c_{e} := \alpha_i\beta_j$. We will show that
$\sum_{e \in E(H)} c_{e}f_{e} = 0$. 

By \eqref{obs:coo} it suffices to show that, for each $v \in V(H)$, we have $\sum_{e \in E(H)} c_e \pi_v(f_e) = 0$. Without loss of generality, consider $v_i \in V(H)$. We have
$$\sum_{e \in E(H)} c_e \pi_{v_i}(f_e) = \sum_{j=1}^t \alpha_i\beta_j\vb_{w_j} = \alpha_i \sum_{j=1}^t \beta_j\vb_{w_j} = 0,$$
by choice of the scalars $\beta_1,\ldots, \beta_t$. This concludes the proof of the claim.
\end{proof}

We will now bound the dimension of the space spanned by the vectors of $\{f_e: e \in E(K_n)\}$.

\begin{claim}\label{cl:dim}
$\vdim(\vspan(\{f_e: e \in E(K_n)\}))\geq (t-1)(n + 1 - t/2)$.
\end{claim}

\begin{proof}
Recall the definition of the graph $G_n$ from Construction~\ref{const:ktt}. We will show that the family of vectors $\{f_e: e \in G_n\}$ are linearly independent. As $|E(G_n)| = (t-1)(n + 1 - t/2)$, this will prove the claim.

Let us suppose that $\Sigma := \sum_{e \in E(G_n)} c_ef_e = 0$. We will show that $c_e = 0$ for every $e \in E(G_n)$. Recall that $V(G_n) = X \cup Y \cup Z$.
First, for $z \in Z$ consider $\pi_z(\Sigma)$. Using \eqref{obs:coo}, we have
$$\pi_z\left(\sum_{e \in E(G_n)} c_ef_e\right) = \sum_{e \in E(G_n)}c_e\pi_z(f_e) = \sum_{y \in Y}c_{yz}\vb_y = 0.$$
As $|Y| = t-1$ and since any $t-1$ vectors of $\vc$ are linearly independent, $c_{yz} = 0$ for all $y \in Y$ and $z \in Z$.

Now suppose, in order to obtain a contradiction, that there exists some $y^* \in Y$ and $x^* \in X$ such that $c_{y^*x^*} \not= 0$. Using \eqref{obs:coo}, for each $y \in Y$ we have \begin{equation}\label{eq:ysum}
0 = \sum_{e \in G_n} c_e \pi_y(f_e) = \sum_{x \in X}c_{yx}\vb_{x}.
\end{equation}
In this case, as $|X| = t$ and any $t$ vectors of $\vc$ are minimally dependent, using \eqref{eq:ysum} we obtain that 
\begin{equation}\label{eq:not0}
c_{y^*x} \not= 0,\hspace{0.2cm} \text{ for all } \hspace{0.2cm} x \in X.
\end{equation}
In addition, as the dependence of the vectors $\mathcal{X}:= \{\vb_x: x \in X\}$ is unique up to scaling by a constant factor, we obtain that for each $y \in Y \setminus \{y^*\}$, there exists $\gamma_{y} \in \mathbb{R}$ such that $c_{yx} = \gamma_{y}c_{y^*x}$, for all $x \in X$ (note that $\gamma_{y}$ may be equal to 0).

Now consider $\pi_x(\Sigma)$, for each $x \in X$. Expanding this out, we obtain 
$$\pi_x(\Sigma) = \sum_{y \in Y}c_{yx}\vb_y + \sum_{x' \in X\setminus \{x\}}c_{xx'}\vb_{x'} = c_{y^* x}\left(\sum_{y \in Y}\gamma_y \vb_y\right) + \sum_{x' \in X\setminus \{x\}}c_{xx'}\vb_{x'}.$$

Then for each $x_1, x_2 \in X$, we have
\begin{equation}\label{eq:ytake}
\frac{\pi_{x_1}(\Sigma)}{c_{y^* x_1}} - \frac{\pi_{x_2}(\Sigma)}{c_{y^*x_2}} =  c_{x_1x_2}\left( \frac{\vb_{x_2}}{c_{y^*  x_1}} - \frac{ \vb_{x_1}}{c_{y^* x_2}}\right) + \sum_{x \in X \setminus \{x_1,x_2\}} \left(\frac{c_{x_1x}}{c_{y^*x_1}} - \frac{c_{x_{2}x}}{c_{y^*x_{2}}}\right)\vb_{x} = 0,
\end{equation}
as by \eqref{obs:coo},  $\pi_x(\Sigma) = 0$, for all $x \in X$. The expression \eqref{eq:ytake} is a linear combination of vectors of $ \mathcal{X} \subseteq \vc$, which by definition are minimally dependent as $|\mathcal{X}| = t$. So this dependence is equal (up to a constant scaling factor) to the dependence between the vectors of $\mathcal{X}$ in \eqref{eq:ysum}. Hence there exists $\eta \in \mathbb{R}$ such that for all $x \in X$, the coefficient of $\vb_x$ in \eqref{eq:ytake} is equal to $\eta c_{y^*x}$. Therefore, by looking at the coefficients of $\vb_{x_1}$ and $\vb_{x_2}$, we obtain that
$$-\frac{c_{x_1x_2}}{c_{y^* x_2}} = \eta c_{y^* x_1} \hspace{0.2cm} \text{ and } \hspace{0.2cm} \frac{c_{x_1x_2}}{c_{y^* x_1}} = \eta c_{y^* x_2},$$
which implies that $c_{x_1x_2} = 0$ (as $c_{y^*x_1}c_{y^* x_2} \not= 0$, by \eqref{eq:not0}).

But now, for any $x \in X$, $\pi_x(\Sigma)$ is a linear combination of $t-1$ vectors of $U$, and hence $c_{xy} = 0$, for all $x \in X$, $y \in Y$ (using \eqref{obs:coo} and the fact that any $t-1$ vectors of $U$ are linearly independent). This contradicts our assumption that $c_{y^* x^*} \not= 0$. 

Hence $c_{xy}=0$ for all $x \in X$ and $y \in Y$ and it remains to show that $c_{xx'}=0$, for all $x,x' \in X$. But now, for any $x \in X$, $$\pi_x(\Sigma) = \sum_{x' \in X\setminus \{x\}}c_{xx'}\vb_{x'} = 0,$$
by \eqref{obs:coo}. As vectors in $U$ are in general position in $\mathbb{R}^{t-1}$, any $t-1$ are linearly independent and hence this expression can only hold if $c_{xx'}=0$ for all $x' \in X \setminus \{x\}$. 

This completes the proof that $\{f_e: e \in E(G_n)\}$ is a family of linearly independent vectors. Hence, $\vdim (\vspan\{f_e : e \in E(F)\}) \geq |E| = (t-1)(n + 1 - t/2)$.
\end{proof}
Given Claims~\ref{cl:dep} and \ref{cl:dim}, we may apply Lemma~\ref{LinearAlgebraBootstrapPer}. This completes the proof of the lemma.
\end{proof}

\begin{proof}[Proof of \Cref{thm:KttWsat}]
The theorem follows immediately from Lemmas~\ref{lem:upp} and \ref{lem:low}.
\end{proof}

\subsection{Determining $\wsat(n,K_{t,t+1})$}
\begin{proof}[Proof of \Cref{corr:Ktt+1}]

For the lower bound, we will show that every graph $G$ which is weakly $(n,K_{t,t+1})$-saturated, contains a proper subgraph $G'$ which is weakly $(n,K_{t,t})$-saturated, and therefore $\wsat(n,K_{t,t+1}) > \wsat(n,K_{t,t})$. Indeed, let $e_1,\ldots,e_t$ be an ordering of $E(K_n)\setminus E(G)$ such that the addition of $e_i$ to $G \cup \{e_1,\ldots,e_{i-1}\}$ creates a copy $H_i$ of $K_{t,t+1}$, such that $e_i \in H_i$. Note that this implies that the addition of each edge also creates a copy $H_i'$ of $K_{t,t}$ with $e_i \in H_i'$. Observe that at the start of the process, $E(H_1) \setminus \{e_1\} \subseteq E(G)$. Therefore $G$ contains a copy of $K_{t,t+1}\setminus \{e\}$, for some edge $e$. In particular, $G$ contains a copy of $K_{t,t}$. So there exists a weakly $(K_n,K_{t,t})$-saturated subgraph $G' \subseteq G$, where $|E(G')| \le |E(G)| - 1$.

For the upper bound, we will construct a weakly $(K_n,K_{t,t+1})$-saturated graph $F_n$ with $(t-1)(n + 1 - t/2) + 1$ edges. Let $X$, $Y \cup \{y^*\}$ and $Z$ be disjoint sets of vertices of cardinality $t$, $t$ and $n-2t$, respectively. Define $F_n$ to be the graph on vertex set $X \cup Y \cup Z \cup \{y^*\}$ and $uv$ is an edge of $F_n$ if and only if either $u \in X$ and $v \in Y \cup \{y^*\}$, $u \in Z$ and $v \in Y$ or $u,v \in X$. That is, $X$ is a clique on $t$ vertices, $Y \cup \{y^*\}$ and $Z$ are both independent sets, and $F_n(X\cup Z,Y)$ is a complete bipartite graph. Observe that $F_n$ has $(t-1)(n + 1 - t/2) + 1$ edges. It is easy to check that $F_n$ is weakly $(K_n,K_{t,t+1})$ saturated: first add the edges from $y^*$ to $Z$, then add the edges between $X$ and $Z$, then the edges within $Z$ and finally the edges within $Y \cup \{y^*\}$.
\end{proof}

\section{Proof of Theorem~\ref{thm:genst}}\label{sec:Kst}

In this section we prove Theorem~\ref{thm:genst}, i.e., we asymptotically determine the value of $\wsat(n, K_{s,t})$ whenever $2\le s < t$ and $n\ge 4t$. We start by describing a construction to give an upper bound for $\wsat(K_n,K_{s,t})$ in the spirit of Construction~\ref{const:ktt}.
\begin{const}\label{const:kst}
We define $V(H_n) =X \cup \{x^*\} \cup Y_1 \cup Y_2 \cup W \cup Z$, where $X$, $Y_1$, $Y_2$, $W$ and $Z$ are disjoint subsets of $[n]$ of cardinality $s-1$, $t-s$, $s-1$, $s-1$ and $n-t-2s+2$, respectively. As $s < t$, note that $Y_1\neq \emptyset$. Let $Y := Y_1 \cup Y_2$.
We have that $uv$ is an edge of $H_n$ if and only if either $u \in X$ and $v \in W \cup Y$, $u = x^*$ and $v\in Y$, $u\in Z$ and $v\in Y_2$ , or $u,v \in Y_1 \cup Y_2$. That is, $Y$ is a clique on $t-1$ vertices and $X \cup \{x^*\}$, $W$ and $Z$ are all independent sets, $H_n[X \cup \{x^*\},Y]$ is a complete bipartite, and all vertices in $W\cup Z$ have degree $s-1$.
\end{const}

\begin{figure}[htbp]
\centering
\includegraphics[width=0.6\textwidth]{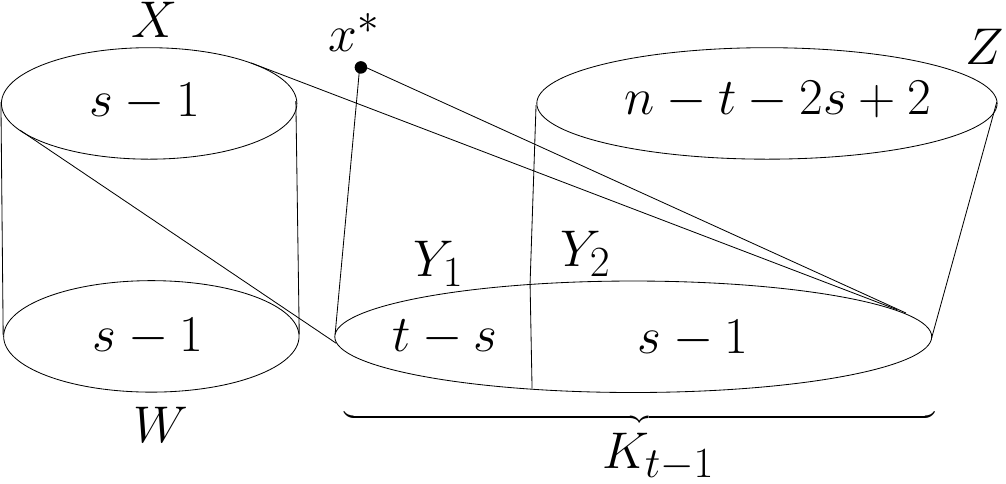}
\caption{The graph $H_n$. The sets $X$, $W$ and $Z$ are independent sets of cardinalities $s-1$, $s-1$ and $n-t - 2s+2$, respectively. All edges between $X$ and $W \cup Y_1 \cup Y_2$, between $x^*$ and $Y_1 \cup Y_2$ and between $Y_2$ and $Z$ are present. Moreover, $Y_1 \cup Y_2$ is a clique on $t-1$ vertices.}
\end{figure}

The upper bound follows from showing that $H_n$ is weakly $(K_n,K_{s,t})$-saturated.

\begin{proposition}\label{prop:KstWsatUpper}
Let $t> s \geq2$ and $n \ge 2(s + t)-3$. Then $$\wsat(n,K_{s,t}) \leq (s-1)(n-s)+\binom t2.$$
\end{proposition}
\begin{proof}
Observe that $H_n$ is $K_{s,t}$-free. Indeed, every vertex in $W\cup Z$ has degree $s-1$, and thus none of them are present in a copy of $K_{s,t}$ (as $s < t$). Now look at the subgraph $H_n(X,Y)$ obtained by removing these vertices. This graph has $s+t-1$ vertices and thus contains no copy of $K_{s,t}$.
	We will show that $H_n$ is weakly $(K_n,K_{s,t})$-saturated by adding the edges of $E(K_n) \setminus E(H_n)$ in such a way that the addition of each edge $e$ creates a copy of $K_{s,t}$ containing $e$.
	
	We are first able to add all edges between $x^*$ and $W$. Indeed, consider $e =x^*w$ for some $w\in W$. Observe that $H_n[X \cup \{x^*\}, Y\cup \{w\}]$ is a copy of $K_{s,t}\setminus \{e\}$ and hence $e$ can be added. 
	
	We next show that we can add all edges between $Y_1$ and $Z$. Consider $e=yz$, where $y\in Y_1$ and $z\in Z$. Observe that $|Y_2\cup \{y\}|=s$ and $|(Y_1\setminus \{y\})\cup X\cup \{x^*,z\})|=t$ and thus $H_n[Y_2\cup \{y\},(Y_1\setminus \{y\})\cup X\cup \{x^*,z\}]$  is a copy of $K_{s,t}\setminus \{e\}$. Hence $e$ can be added. 
	
	We can now add all edges between $W$ and $Z$. Let $e= wz$, where $z\in Z$ and $w\in W$. Then $|X\cup \{z\}|=s$ and $|Y\cup \{w\}|=t$,  and thus $H_n[X\cup \{z\},Y \cup \{w\}]$ is a copy of $K_{s,t}\setminus \{e\}$. Hence $e$ can be added.

	For the next step, we show that we can add all edges between $Y$ and $W$. Let $e= wy$ where $w\in W$ and $y\in Y$. Let $\tilde{Y}$ be a subset of $Y\setminus \{y\}$ of size $s-1$, and $\tilde{Z}$ a subset of $Z$ of size $t-1$. Then $|\tilde{Y}\cup \{w\}|=s$ and $|\tilde{Z}\cup \{y\}|=t$, and thus $H_n[\tilde{Y}\cup \{w\},\tilde{Z}\cup \{y\}]$ is a copy of $K_{s,t}\setminus \{e\}$. Hence $e$ can be added.

    Next we can add all edges within $W$. Let $e = uv$, where $u, v \in W$ and let $Z'$ be a subset of $Z$ of size $s-1$. Observe that $H_n[Z' \cup \{u\}, Y \cup \{v\}]$ is a copy of $K_{s,t}\setminus \{e\}$. Hence $e$ can be added.
	
	It remains to show that any edge $e= uv$, where $u,v \in X\cup \{x^*\}\cup Z$ can be added. As all edges outside of $X\cup\{x^*\}\cup Z$ have been added, and $|Y\cup W|=s+t-2$, then $H_n[W\cup \{u\},Y\cup\{v\}]$ is a copy of  $K_{s,t}\setminus \{e\}$. This completes the proof.
\end{proof}

For a lower bound for $\wsat(K_{n},K_{s,t})$, inspired by the argument in~\cite{MoshShapira} (see also \Cref{appen:BipInBip}), we can show the following.

\begin{proposition}\label{prop:KstWsatLower}
Let $t> s \ge 2$ and $n \ge 3t-3$. Then $$\wsat(n,K_{s,t}) \geq (s-1)(n-t+1)+\binom t2.$$
\end{proposition}

\begin{proof}
Let $G\subseteq K_{n}$ be a graph on $n$ vertices, with $\text{wsat}(K_{n},K_{s,t})$ edges, and assume that $G$ is weakly $(K_n,K_{s,t})$-saturated with the corresponding ordering of the missing edges $\{e_1,\dots, e_h\}$ and such that $C_i$ is a copy of $K_{s,t}$ created by adding $e_i$. 
 Let $G'$ be a graph obtained by $G$ as follows. $V(G')=V(G)\cup X$, where $X$ is a set of size $t-s$ disjoint from $V(G)$, and $E(G')=E(G)\cup \{vu \mid v\in X,\ u\in V(G) \}$.  Then $G'$ is a graph with $n+t-s$ vertices and has $\text{wsat}(K_{n},K_{s,t})+n(t-s)$ edges. Now, note that $G'$ is weakly $(K_{n+t-s},K_{t,t})$-saturated. Indeed, the edges $e_i$, $1\leq i\leq h$, can still be added using $C_i$ together with the $t-s$ new vertices that were added. The edges inside $X$  can then be added using any $2t-2$ vertices from $V(G)$. By the minimality of $\text{wsat}$, we have that $|E(G')|\geq \text{wsat}(K_{n+t-s},K_{t,t})= (t-1)(n-s+1)+\binom t2$ (by \Cref{thm:KttWsat}). All together,   we obtain $\text{wsat}(K_{n},K_{s,t})+n(t-s)\geq (t-1)(n-s+1)+\binom t2$, that is, $\text{wsat}(K_{n},K_{s,t})\geq (n-t+1)(s-1)+\binom t2$.
\end{proof}

Note that this lower bound matches the upper bound from \Cref{prop:KstWsatUpper} when $s=t-1$, which is the content of \Cref{corr:Ktt+1}. It is also worth mentioning that the upper and lower bounds differ only by $(t-s-1)(s-1)$, and thus \Cref{prop:KstWsatUpper} and \Cref{prop:KstWsatLower} imply Theorem~\ref{thm:genst}.

\section{Conclusion}\label{sec:con}
In this paper, for $n \ge 3t-3$  we have exactly determined $\wsat(n,K_{t,t})$ (see Theorem~\ref{thm:KttWsat}) and $\wsat(n,K_{t,t+1})$ (see Corollary~\ref{corr:Ktt+1}).
Now that $\wsat(n,K_t)$ and $\wsat(n,K_{t,t})$ are known, the next natural question is to consider balanced multipartite graphs. Let $K_{t}^k$ denote the complete multipartite graph containing $k$ parts each of size $t$. A generalisation of our construction for $\wsat(n,K_{t,t})$ yields a plausibly tight upper bound for $\wsat(n,K_{t}^k)$. We are curious as to whether this could be best possible for large $n$  (note that for our constructions to be weakly saturated we need a lower bound on $n$).

\begin{const}\label{const:kkt}
Let $F_n^{k,t}$ be an $n$-vertex graph on vertex set $X \cup Y \cup Z$, where $X \cup Y$ contains a complete $k$ partite graph with vertex classes $X = C_1$, $Y= C_2,\ldots,C_k$ where $|C_i| = t$ for $i \le k-1$ and $|C_k| = t-1$; $X$ induces a clique on $t$ vertices; and, $Z$ is an independent set of size $n - tk + 1$ such that $F_n^{k,t}[Y,Z]$ is a complete bipartite graph.
\end{const}

\begin{figure}[htbp]
\centering
\includegraphics[width=0.3\textwidth]{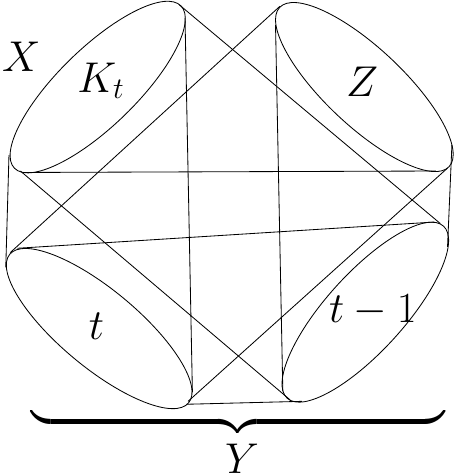}
\caption{The graph $F_n^{3,t}$. The set $X$ is a clique on $t$ vertices while $Z$ is an independent set on $n-3t +1$ vertices. Additionally, $Y$ induces a complete bipartite graph on vertex classes of sizes $t$ and $t-1$. Every edge between $Y$ and $X \cup Z$ is present.}
\end{figure}

It is not difficult to check that the graph $F_{n}^{k,t}$ is $K^{k}_{t}$-weakly-saturated (the details are left to the reader) which implies
$\wsat(K_n, K_{t}^{k}) \leq |E(F_n^{k,t})|$ for $n\geq (k+1)t - 2$.
Observe that $K_r$ is isomorphic to $K_1^r$ and so, by Theorem~\ref{thm:clique}, we have $\wsat(K_n, K_r) = |E(F_{n}^{r,1})|$.

\begin{ques}
Is there $n_0$ such that for all $n \ge n_0$, we have 
$\wsat(K_n, K_{t}^{k}) = |E(F_n^{k,t})|$?
\end{ques}
 Let us now turn our attention to unbalanced bipartite graphs. For $2 \le s < t$ we have provided an asymptotically tight bound on $\wsat(n,K_{s,t})$ (see Theorem~\ref{thm:genst}). 
It would be interesting to pin this value down precisely. We wonder if Construction~\ref{const:kst} is best possible for large $n$, i.e. if Proposition~\ref{prop:KstWsatUpper} is tight for large $n$.

\begin{ques}
Is there some $n_0$ such that for all $n \ge n_0$ and $t >  s +1 \ge 2$, we have $$\wsat(n,K_{s,t}) = (s-1)(n-s) + \binom{t}{2}?$$
\end{ques}

In particular, we believe that the lower bound given by Proposition~\ref{prop:KstWsatLower} is not tight in general; it seems that by analysing the process more carefully it could be possible to add fewer extra edges to convert the $K_{s,t}$ process into the $K_{t,t}$ process.

Although Corollary~\ref{cor:rel} reveals a relationship between $\wsat(n,K_t)$ and $\wsat(K_{\ell,m},K_{t,t})$, it is not obvious how to use knowledge of the weak saturation numbers of $K_{s,t}$ within $K_{\ell,m}$ to bound $\wsat(n,K_{s,t})$. It is plausible (but wrong) to believe that for $s<t$, we have $\wsat(n,K_{s,t}) > \wsat(K_{\ell,m}, K_{s,t})$ (where $\ell + m = n$). In $K_n$, the process must eventually add more edges (those not respecting the bipartition), but we also have more freedom to choose the edges of the ``starting" graph, and it is not obvious that we cannot ``save" some edges.

We would like to emphasize that comparing the bounds given by \Cref{prop:KstWsatUpper} and \Cref{prop:KstWsatLower} for $\wsat(n,K_{s,t})$ to the value of $\text{wsat}(K_{\ell,m},K_{s,t})$ (given by Theorem~\ref{thm:wsat-imp}) where $\ell+m=n$, shows that $\text{wsat}(F,H)$ does not obviously compare to $\text{wsat}(F',H)$ when $F'\subseteq F$.
Indeed, by \Cref{prop:KstWsatUpper}, when $t>3s$ we get $\wsat(K_n,K_{s,t}) < \wsat(K_{\ell,m},K_{s,t})$, and by \Cref{prop:KstWsatLower}, for $2s>t$ we get $\wsat(K_n,K_{s,t}) \geq \wsat(K_{\ell,m},K_{s,t})$.

We therefore believe it is interesting to consider the following. 
 
\begin{ques}
Let $2 \le s < t$ and $k + \ell = n$, for large $n$. When do we have
$$\wsat(n,K_{s,t}) > \wsat(K_{\ell,m},K_{s,t})?$$
\end{ques}

It would also be very interesting to determine weak saturation numbers of general unbalanced multipartite graphs. Let $K_{a_1,\ldots,a_k}$ denote the complete multipartite graph with parts of size $a_1,\ldots,a_k$.

\begin{ques}
What is $\wsat(n,K_{a_1,\ldots,a_k})$?
\end{ques}

We remark that for $a_1, \ldots, a_{k-1} \le a_k$, an argument analogous to the one in Proposition~\ref{prop:KstWsatLower} would give a lower bound on 
$\wsat(n, K_{a_1, \ldots, a_k})$ based on $\wsat(n + (a_k - a_{k-1}) + \ldots + (a_k - a_1), K_{a_k, \ldots, a_k})$.
In general, $\wsat(n, K_{a_1, \ldots, a_k})$ can always be lower bounded based on $\wsat(n + (s_1 - a_1) + \ldots + (s_k - a_k), K_{s_1,\ldots, s_k})$ as long as $a_i \le s_i$ for $i \in [k]$. Although this method could give a good lower bound for an asymptotic result, we do not believe that it would give a tight bound.



\section*{Acknowledgment}
This work began while the authors were visiting the Instituto Nacional de Matem\'atica Pura e Aplicada in Rio de Janeiro as part of the Graphs@IMPA thematic program.
We are very grateful to IMPA and the organizers for the support to be able to attend the event and the great working environment.

\newcommand{\cpc}{Combin. Probab. Comput. }
\newcommand{\discrete}{Discrete Math. }
\newcommand{\eur}{European J.~Combin. }
\newcommand{\jcta}{J.~Combin. Theory Ser.~A }
\newcommand{\jctb}{J.~Combin. Theory Ser.~B }
\newcommand{\rsa}{Random Structures Algorithms }

\appendix
\section{Proof of Theorem~\ref{thm:wsat-imp}}\label{appen:BipInBip}

Here we will prove Theorem~\ref{thm:wsat-imp} by generalising the argument from~\cite{MoshShapira} for Theorem~\ref{thm:MS}.

\begin{proof}[Proof of Theorem~\ref{thm:wsat-imp}]

We start with the upper bound. Let $X_1,X_2,X_3,Y_1,Y_2,Y_3$ be disjoint sets of vertices such that $|X_1|=|Y_1|=s-1$, $|X_2|=|Y_2|=t-s$, $|X_3|=\ell-t+1$, $|Y_3|=m-t+1$. Denote $X=X_1\cup X_2\cup X_3$ and $Y=Y_1\cup Y_2\cup Y_3$. Note that $|X|=\ell$ and $|Y|=m$. Let $G_0\subseteq K_{\ell,m}$ be a bipartite graph on vertex set $X\cup Y$ where $vu\in E(G_0)$ if and only if $v\in X_1$ and $u\in Y$, or $v\in Y_1$ and $u\in X$, or $v\in X_2$ and $u\in Y_2$. Note that  $|E(G_0)|=(\ell+m)(s-1)+(t-s)^2$. We will show that $G_0$ is weakly $(K_{\ell,m},K_{s,t})$-saturated. 

Since $G_0[X_1\cup X_2,Y_1\cup Y_2]$, $G_0[X,Y_3]$ and $G_0[X_3,Y]$  are all complete bipartite graphs, we only need to show how to add the edges $vu$ when $v\in X_3$ and $u\in Y_2\cup Y_3$, and when $v\in Y_3$ and $u\in X_2$.

We start with adding the edges $vu$ when $v\in Y_3$ and $u\in X_2$. Since $|\{u\}\cup X_1|=s$, $|\{v\}\cup X_1\cup X_2|=t$, and  $G_0[\{u\}\cup X_1,\{v\}\cup Y_1\cup Y_2]$ is a copy of $K_{s,t}$ minus one edge, we can add the missing edge $vu$. In a similar fashion, we can add edges $vu$ when $v \in Y_2$ and $u \in X_3$. Now we only need to add the edges $vu$ for $v\in X_3$ and $u\in Y_3$. Since $|\{v\}\cup X_1|=s$, $|\{u\}\cup Y_1\cup Y_2|=t$, and  $G_0[\{v\}\cup X_1,\{u\}\cup Y_1\cup Y_2]$ is a copy of $K_{s,t}$ missing $vu$, we can add the edge $vu$. This shows that $G_0$ is weakly $(K_{\ell,m},K_{s,t})$-saturated and completes the proof of the upper bound.

For the lower bound, we will apply Theorem~\ref{thm:MS}. Let $G\subseteq K_{\ell,m}$ be a weakly $(K_{\ell,m},K_{s,t})$-saturated graph with $\wsat(K_{\ell,m},K_{s,t})$ edges. Let $\{e_1,\dots, e_h\}$ be an ordering of $E(K_{\ell,m}\setminus G)$ such that $C_i$ is a copy of $K_{s,t}$ in $G \cup \{e_1,\ldots,e_i\}$ containing $e_i$.

Denote by $X_1$ the vertex set of $G$ of size $\ell$, and by $Y_1$ the vertex set of $G$ of size $m$. Let $X_2, Y_2$ be two disjoint sets (also disjoint from $X_1,Y_1$), each of size $t-s$. Let $G'$ be the bipartite graph with parts $X=X_1\cup X_2$ and $Y=Y_1\cup Y_2$, obtained from $G$ as follows. $E(G')=E(G)\cup \{vu \mid v\in X_2,\ u\in Y_1 \}\cup \{vu \mid v\in Y_2,\ u\in X_1 \}$.  Then $G'$ is a bipartite graph with parts of size $\ell+t-s$ and $m+t-s$, and has $\text{wsat}(K_{\ell,m},K_{s,t})+(\ell+m)(t-s)$ edges. Now, note that $G'$ is weakly $(K_{\ell+t-s,m+t-s},K_{t,t})$-saturated. Indeed, the edges $e_i$, $1\leq i\leq h$, can still be added using $C_i$ together with the $t-s$ new vertices that were added to one of the sides. The edges between $X_2$ and $Y_2$ can then be added using $t-1$ vertices from $X_1$ and $t-1$ vertices from $Y_1$. By the minimality of $\text{wsat}$, we have that $$|E(G')|\geq \text{wsat}(K_{\ell+t-s,m+t-s},K_{t,t})= (\ell+t-s)(m+t-s)-(\ell-s+1)(m-s+1).$$ As  $|E(G')| = \text{wsat}(K_{\ell,m},K_{s,t})+(\ell+m)(t-s)$, we obtain $$\text{wsat}(K_{\ell,m},K_{s,t})\geq \ell m-(\ell+t-s)(m+t-s)+(t-s)^2,$$ as required.
\end{proof}

\begin{thebibliography}{99}



\bibitem{AlonWeak1985}
N. Alon:
{\em An extremal problem for sets with applications to graph theory\/},
\jcta {\bf 40} (1985), 82--89.






\bibitem{BBMR}
J. Balogh,  B. Bollob\'as, R. Morris and O. Riordan:
{\em Linear algebra and bootstrap percolation\/},
\jcta {\bf 119} (2012), 1328--1335.

\bibitem{BaloghPete1998}
J. Balogh and G. Pete:
{\em Random disease on the square grid\/},
\rsa {\bf 13} (1998), 409--422.




\bibitem{bol2f}
B. Bollob\'{a}s:
{\em On generalized graphs\/},
Acta Math. {\bf 16} (1965), 447--452.


\bibitem{BollobasWSat}
B. Bollob\'as:
{\em Weakly $k$-saturated graphs\/},
Beitr\"age zur Graphentheorie (1968) 25--31.

\bibitem{BS02}
M. Borowiecki and E. Sidorowicz:
{\em Weakly $P$-saturated graphs\/},
Discuss.~Math. Graph Theory {\bf 22} (2002), 17--29.








\bibitem{EFT91}
P. Erd\H os, F\"uredi and Z. Tuza:
{\em Saturated r-uniform hypergraphs\/},
\discrete {\bf 98} (1991), 95--104.



\bibitem{FaudreeSurvey}
J. R. Faudree,  R. J. Faudree and J. R. Schmitt:
{\em A survey of minimum saturated graphs\/},
Electron.~J.~Comb. {\bf 1000} (2011), 19--29.

\bibitem{FaudreeWsatMultipleCopies}
R. J. Faudree and R. J. Gould:
{\em Weak saturation numbers for multiple copies\/}, 
\discrete {\bf 336} (2014), 1--6.

\bibitem{FaudreeWsatSparseGraphs}
R. J. Faudree, R. J. Gould and M. S. Jacobson:
{\em Weak saturation numbers for sparse graphs\/}
Discuss.~Math. Graph Theory {\bf 33} (2013), 677--693.


\bibitem{FranklWeakSat1982}
P. Frankl:
{\em An extremal problem for two families of sets\/},
\eur {\bf 3} (1982), 125--127.



\bibitem{KalaiWeakSat1985}
G. Kalai:
{\em Hyperconnectivity of graphs\/},
Graphs Combin. {\bf 1} (1985), 65--79.


\bibitem{KalaiWeakSat1984}
G. Kalai:
{\em Weakly saturated graphs are rigid\/},
North-Holland Math.~Stud. {\bf 87} (1984), 189--190.



\bibitem{LovaszMultiLinear}
L. Lov\'asz:
{\em Flats in matroids and geometric graphs\/},
Combinatorial Surveys (1977), 45--86.


\bibitem{MorrisonNoelHypergraph}
N. Morrison and J. A. Noel:
{\em A Sharp Threshold for Bootstrap Percolation in a Random Hypergraph\/},
available as arXiv:1806.02903.


\bibitem{MorrisonNoelHypercube}
N. Morrison and J. A. Noel:
{\em Extremal bounds for bootstrap percolation in the hypercube\/},
\jcta {\bf 156} (2018), 61--84.

	
\bibitem{MoshShapira}
G. Moshkovitz and A. Shapira:
{\em Exact bounds for some hypergraph saturation problems\/},
\jctb {\bf 111} (2015), 242--248.
	

	
\bibitem{PikhurkoWeakSatHypergraphs1}
O. Pikhurko:
{\em Uniform families and count matroids\/},
Graphs Combin. {\bf 17} (2001), 729--740.
	
\bibitem{PikhurkoWeakSatHypergraphs}
O. Pikhurko:
{\em Weakly saturated hypergraphs and exterior algebra\/},
\cpc {\bf 10} (2001), 435--451.
	




\bibitem{Sem97}
G. Semani\v sin:
{\em On some variations of extremal graph problems\/},
Discuss. Math. Graph Theory {\bf 17} (1997), 67--76.

\bibitem{Sid07}
E. Sidorowicz:
{\em Size of weakly saturated graphs\/},
\discrete {\bf 307} (2007), 1486--1492.

\bibitem{Tuza92}
Z. Tuza:
{\em Asymptotic growth of sparse saturated structures is locally determined\/},
\discrete {\bf 108} (1992), 397--402.

\bibitem{Tuza88}
Z. Tuza:
{\em Extremal problems on saturated graphs and hypergraphs\/},
Ars Combin. {\bf 25} (1988), 105--113.

\end{thebibliography}
\end{document}